\documentclass[12pt]{article}
\usepackage{amsmath,amssymb,amsthm, amsfonts}
\usepackage{hyperref}
\usepackage{graphicx}
\usepackage{url}
\usepackage{dsfont} 
\usepackage{color}
\definecolor{red}{rgb}{1,0,0}

\definecolor{blue}{rgb}{0,0,1}

\definecolor{green}{rgb}{0,.6,0}

\usepackage{float}

\usepackage{tikz}

\newtheorem{thm}{Theorem}[section]
\newtheorem{cor}[thm]{Corollary}
\newtheorem{lem}[thm]{Lemma}
\newtheorem{prop}[thm]{Proposition}

\newtheorem{obs}[thm]{Observation}

\newtheorem{prob}[thm]{Problem}

\theoremstyle{definition}
\newtheorem{rem}[thm]{Remark}

\theoremstyle{definition}
\newtheorem{defn}[thm]{Definition}

\theoremstyle{definition}
\newtheorem{ex}[thm]{Example}

\newcommand{\G}{\mathcal{G}}

\newcommand{\Z}{\operatorname{Z}}

\newcommand{\bit}{\begin{itemize}}
\newcommand{\eit}{\end{itemize}}
\newcommand{\ben}{\begin{enumerate}}
\newcommand{\een}{\end{enumerate}}
\newcommand{\beq}{\begin{equation}}
\newcommand{\eeq}{\end{equation}}
\newcommand{\bea}{\begin{eqnarray}} 
\newcommand{\eea}{\end{eqnarray}}
\newcommand{\bpf}{\begin{proof}}
\newcommand{\epf}{\end{proof}\ms}
\newcommand{\bmt}{\begin{bmatrix}}
\newcommand{\emt}{\end{bmatrix}}
\newcommand{\ms}{\medskip}

\newcommand{\noi}{\noindent}

\newcommand{\beqs}{\begin{equation*}} 
\newcommand{\eeqs}{\end{equation*}}
\newcommand{\beas}{\begin{eqnarray*}}
\newcommand{\eeas}{\end{eqnarray*}}

\newcommand{\up}[1]{^{(#1)}}
\newcommand{\upc}[1]{^{[#1]}}
\newcommand{\floor}[1]{\lfloor #1 \rfloor}
\newcommand{\calf}{\mathcal{F}}
\newcommand{\calm}{\mathcal{M}}

\newcommand{\zf}{\operatorname{\lfloor \operatorname{Z} \rfloor}}
\newcommand{\zpf}{\floor{\operatorname{Z}_{+}}}
\newcommand{\zp}{\operatorname{Z}_{+}}

\newcommand{\pt}{\operatorname{pt}}
\newcommand{\ptp}{\operatorname{pt}_{+}}

\newcommand{\ptzpf}{\operatorname{pt}_{\zpf}}

\newcommand{\throt}{\operatorname{th}}
\newcommand{\thz}{\operatorname{th_{\Z}}}
\newcommand{\thzf}{\operatorname{th_{\zf}}}
\newcommand{\thzpf}{\operatorname{th_{\zpf}}}

\newcommand{\thp}{\operatorname{th}_{+}}
\newcommand{\thr}[1]{\operatorname{th}(#1)}

\title{Various Characterizations of Throttling Numbers}
\author{Joshua Carlson\thanks{Dept.~of Mathematics and Statistics, Williams College, Williamstown, MA, USA (jc31@williams.edu)} \and 
J\"urgen Kritschgau\thanks{Dept.~of Mathematics, Iowa State University, Ames, IA, USA (jkritsch@iastate.edu) Research is supported by NSF grant DMS-1839918} }

\date{\today}

\begin{document}
\maketitle

\begin{abstract} 
Zero forcing can be described as a graph process that uses a color change rule in which vertices change white vertices to blue. The throttling number of a graph minimizes the sum of the number of vertices initially colored blue and the number of time steps required to color the entire graph. Positive semidefinite (PSD) zero forcing is a commonly studied variant of standard zero forcing that alters the color change rule. This paper introduces a method for extending a graph using a PSD zero forcing process. Using this extension method, graphs with PSD throttling number at most $t$ are characterized as specific minors of the Cartesian product of complete graphs and trees. A similar characterization is obtained for the minor monotone floor of PSD zero forcing. Finally, the set of connected graphs on $n$ vertices with throttling number at least $n-k$ is characterized by forbidding a finite family of induced subgraphs. These forbidden subgraphs are constructed for standard throttling.
\end{abstract}

\noi {\bf Keywords} Zero forcing, propagation time, throttling, minor monotone floor, positive semidefinite, forbidden subgraphs, color change rule

\noi{\bf AMS subject classification} 05C57, 05C15, 05C50



\begin{section}{Introduction}\label{intro}
Consider a process that requires initial resources. Intuitively, changing the initial resources can change the amount of time it takes to complete the process. For a simple example, consider the process of spreading information. The set of people who initially know the information are the initial resources and the time it takes for everyone to know the information is the completion time. The general idea of throttling is to balance the amount of initial resources with the completion time in order to make the process as efficient as possible. Many of these kinds of processes can be described in the context of graph theory. An example of this is zero forcing, a process in which an initial set of blue vertices and a color change rule is used to progressively change the color of all vertices in the graph to blue. Zero forcing was introduced in \cite{AIM} as a way to bound the maximum nullity of a family of matrices corresponding to a given graph. Throttling for zero forcing was first studied by Butler and Young in \cite{BY13Throt}. Recently, the study of throttling has been expanded to include many variations of zero forcing in \cite{powerdomthrot, JCThrot, PSD} and cops and robbers in \cite{CRthrottle2, CRthrottle}.

The graphs in this paper are simple, finite, and undirected. If $G$ is a graph, $V(G)$ and $E(G)$ denote the sets of vertices and edges of $G$ respectively. The edges of a graph can be denoted as subsets or by juxtaposition of the endpoints (i.e., $uv$ is an edge if $\{u, v\} \in E(G)$). The order of a graph $G$ is $|G| = |V(G)|$. The notation $G \leq H$ is used if $G$ is a subgraph of $H$. If $G \leq H$ and $|G| = |H|$, then $G$ is a spanning subgraph of $H$. If $G$ is a minor of $H$, write $G \preceq H$. In the case that $H \leq G$, $H \leq G$ and $|H| = |G|$, or $H \preceq G$, it is said that $G$ is a supergraph, spanning supergraph, or major of $H$ respectively. For a graph parameter $p$ whose range is well-ordered, the \emph{minor monotone floor of $p$} is defined as $\floor{p}(G) = \min\{p(H) \ | \ G \preceq H\}$.

In \cite{JCThrot}, definitions are given that generalize throttling for zero forcing and many of its variants. In a graph whose vertices are white or blue, an \emph{(abstract) color change rule} for zero forcing is a set of conditions that allow a vertex $u$ to force a white vertex $w$ to become blue. If $R$ is the color change rule, it is said that $u$ $R$ forces $w$ to become blue. The $R$ can be dropped if the rule is clear context and forces are denoted as $u \rightarrow w$. Let $R$ be a given color change rule and let $G$ be a graph with $B \subseteq V(G)$ colored blue and $V(G) \setminus B$ colored white. A \emph{chronological list of $R$ forces of $B$} is an ordered list of valid $R$ forces that can be performed consecutively in $G$ resulting in a coloring in which no more $R$ forces are possible. After each force in a chronological list has been performed, the resulting set of blue vertices in $G$ is an \emph{$R$ final coloring of $B$}. The set $B$ is an \emph{$R$ forcing set of $G$} if $V(G)$ is an $R$ final coloring of $B$. The minimum size of an $R$ forcing set of $G$ is the \emph{$R$ forcing number} of $G$ and is denoted as $R(G)$. 

The set of forces in a particular chronological list of $R$ forces of $B$ is called a \emph{set of $R$ forces of $B$}. Sets of forces are used to define propagation time. If $\calf$ is a set of $R$ forces of $B$, then $\calf\up{0} = B$. For each $t \geq 0$, $\calf\up{t+1}$ is the set of vertices $w$ such that $(u \rightarrow w) \in\calf$ for some $u \in V(G)$ and $(u\rightarrow w)$ is a valid $R$ force given that $\bigcup_{i = 0}^t \calf\up{i}$ is colored blue and $V(G) \setminus \bigcup_{i = 0}^t \calf\up{i}$ is colored white. The smallest $t'$ such that $\bigcup_{i=0}^{t'} \calf\up{i} = V(G)$ is the \emph{$R$ propagation time of $\calf$ in $G$} and is denoted as $\pt_R(G; \calf)$. Note that if the forces in $\calf$ do not eventually color every vertex in $V(G)$ blue, then $\pt_R(G; \calf) = \infty$. The propagation process of $\calf$ breaks $\calf$ into time steps. At time $t=0$, $B$ is blue and $V(G) \setminus B$ is white. For each $t \geq 1$, time step $t$ starts at time $t-1$ with $\bigcup_{i=0}^{t-1} \calf\up{i}$ colored blue and performs every possible force in $\calf$ at that time (transitioning to time $t$ by coloring $\calf\up{t}$ blue). For a set $B \subseteq V(G)$, the \emph{$R$ propagation time of $B$ in $G$} is defined as $\pt_R(G; B) = \min\{ \pt_R(G; \calf) \ | \ \calf \text{ is a set of }R\text{ forces of }B\}.$ The \emph{$R$ throttling number of $B$} is $\throt_R(G; B) = |B| + \pt_R(G; B)$ and the \emph{$R$ throttling number of a graph $G$} is $\throt_R(G) = \min \{\throt_R(G; B) \ | \ B \subseteq V(G)\}$.

The \emph{(standard) zero forcing color change rule}, denoted $\Z$, is that a blue vertex $u$ can force a white vertex $w$ to become blue if $w$ is the only white neighbor of $u$. If $G$ is a graph, $\Z(G)$ is the zero forcing number of $G$ and $\Z$ forces are also called ``standard" forces. It is shown in \cite{Parameters} that the minor monotone floor of $\Z$, denoted $\zf$, can be described as a zero forcing parameter. In this variant, the initial blue vertices are considered ``active" and vertices become inactive after they perform a force. The \emph{$\zf$ color change rule} is that an active blue vertex $u \in V(G)$ can force a white vertex $w \in V(G)$ to become blue if $u$ has no white neighbors in $V(G) \setminus \{w\}$. Note that if $w$ is the only white neighbor of $u$ in $G$, then the force $u \rightarrow w$ is a standard force. If $u$ has no white neighbors in $G$, the force $u$ is said to force $w$ by ``hopping". So a $\zf$ force is either a $\Z$ force or a force by hopping. In \cite{JCThrot}, certain minors of the Cartesian product of a complete graph and a path are shown to characterize graphs $G$ with $\thz(G)$ (and $\thzf(G)$) at most $t$ for any fixed positive integer $t$. The proofs of these characterizations use a method of extending a given graph $G$ into a major of $G$ according to a set of forces. 

Suppose $G$ is a graph and $B \subseteq V(G)$ is the set of vertices in $G$ that are colored blue. Let $W_1, W_2, \ldots, W_k$ be the sets of white vertices in the $k$ connected components of $G-B$. The \emph{positive semidefinite (PSD) zero forcing color change rule}, denoted $\zp$, is that a blue vertex $u \in B$ can force a white vertex $w$ to become blue if $w$ is the only white neighbor of $u$ in $G[B \cup W_i]$ for some $1 \leq i \leq k$. PSD zero forcing can be thought of as standard zero forcing in each $G[B \cup W_i]$ and $\zp$ forces are also called PSD forces. PSD propagation and throttling are studied in \cite{PSDproptime} and \cite{PSD} before the introduction of the general definitions in \cite{JCThrot}. Consistent with the original literature, $\zp$ propagation and throttling are denoted as $\ptp$ and $\thp$ respectively.

Note that for a graph $G$ and subset $B \subseteq V(G)$, the only thing that distinguishes two sets of standard (or PSD) forces of $B$ is the vertices that are performing the forces. In other words, if $R$ is either the standard or PSD color change rule and $\calf$ is a set of $R$ forces of $B$, then the sets $\{\calf\up{i} \ | \ 0 \leq i \leq \pt_R(G; B) \}$ are unique to the choice of $B$ and do not depend on $\calf$. For this reason, we often use the conventional notation $B\up{t} = \calf\up{t}$ and $B\upc{t} = \bigcup_{i = 0}^t \calf\up{i}$ in the context of standard or PSD zero forcing. It is important to note that this convention is not possible for $\zf$ forcing because there are usually many choices for hopping. This fact motivates the general definition of throttling in \cite{JCThrot}.

Let $\calf$ be a set of forces of $B$. A maximal sequence of vertices $ v_1, v_2, \ldots , v_{\ell} $ with $(v_{i} \rightarrow v_{i +1}) \in \calf$  is called a \emph{forcing chain of $\calf$}. For each $u \in B$, $V_u$ is the set of vertices $w$ such that there is a forcing chain containing $u$ and $w$. If $\calf$ is a set of PSD forces of $B$, then the subgraph $T_u(\calf) = G[V_u]$ is a \emph{forcing tree of $\calf$}. If $k$ is a positive integer, a \emph{$k$-ary tree} is a rooted tree such that every vertex either has $k$ children or is a leaf. 

In Section 2, we define an extension technique for PSD zero forcing and use it to characterize all graphs $G$ with $\thp(G) \leq t$ as certain minors of the Cartesian product of a complete graph and a $k$-ary tree. Section 3 gives a similar characterization for a variant of PSD zero forcing that uses hopping (called the minor monotone floor of $\zp$). Finally, standard and PSD throttling numbers are characterized using forbidden induced subgraphs in Section \ref{sectFSP}.

\end{section}
\begin{section}{Throttling Positive Semidefinite Zero Forcing}\label{PSDcharSection}
In this section, a technique is given for extending a graph using a PSD zero forcing process that generalizes the extension for standard zero forcing in \cite[Definition 3.12]{JCThrot}. This extension is used to obtain a characterization for all graphs $G$ that satisfy $\thp(G) \leq t$ for any fixed positive integer $t$. In order to describe the PSD extension process, it is useful to consider a new perspective of PSD propagation.

The color change rule for PSD zero forcing requires breaking a graph into components and performing forces in each component individually. Suppose $G$ is a graph, $\calf$ is a set of PSD forces of a PSD zero forcing set $B \subseteq V(G)$, and $\ptp(G; B) = \ptp(G; \calf)$. Throughout the literature on PSD zero forcing (see \cite{PSD, PSDproptime}), each time step $t$ in the PSD propagation process of $B$ is visualized as follows. Start by removing the current set $B\upc{t-1}$ of blue vertices in $G$ to obtain components $W_1, W_2, \ldots, W_k$ of $G - B\upc{t-1}$. Then for each $1 \leq i \leq k$, perform one time step of standard zero forcing in $G[B\upc{t-1} \cup W_i]$. Finally, update the set of blue vertices in $G$ to $B\upc{t}$. Removing all blue vertices in $G$ at each time step can be misleading because it seems like the vertices that became blue in the previous time step could potentially perform a force in any of the white components. Example \ref{OldIsBadEx} demonstrates that this is not the case.

\begin{ex}\label{OldIsBadEx}
Let $G = P_5$ be the path with vertices labeled in order as $v_1$, $v_2$, $v_3$, $v_4$, and $v_5$. Let $B\upc{0}= v_3$ and consider the set of PSD forces $\calf = \{v_3 \rightarrow v_4, v_3 \rightarrow v_2, v_4 \rightarrow v_5, v_2 \rightarrow v_1$\}. Note that $G - B\upc{0}$ has two components with vertices $W_1 = \{v_4, v_5\}$ and $W_2 = \{v_1, v_2\}$ respectively. The forces $v_3 \rightarrow v_2$ and $v_3 \rightarrow v_4$ occur in the first time step and $B\upc{1} = \{v_2, v_3, v_4\}$. In the second time step, $G - B\upc{1}$ has two components with vertices $W_3 = \{v_1\}$ and $W_{4} = \{v_5\}$. Note that in $G[B\upc{1} \cup W_4]$, $v_2$ has no white neighbors. Also $v_4$ has no white neighbors in $G[B\upc{1} \cup W_3]$. So $v_2$ cannot force in $W_4$ and $v_4$ cannot force in $W_3$.
\end{ex}

In general, if a blue vertex $v$ is forced in component $W$, then $v$ cannot perform a force in any future component that is not contained in $W$. This means that it is more natural to think of the PSD propagation process in the following way. In the first time step, the set $B$ of blue vertices is removed from the graph $G$, a copy of $G[B]$ is re-attached to each component of $G-B$, and one time step of standard zero forcing is applied to each of the resulting graphs. Then, in each subsequent time step, this process is repeated on each of the smaller graphs. Note that we no longer consider an updated set of blue vertices in the original graph $G$. Also, each time step can be thought of as applying the first time step to a reduced version (i.e., induced subgraph) of the graphs obtained in the previous step. The next example illustrates this process and is used throughout this section.

\begin{ex}\label{redPSDEx}
Suppose $G$ is the graph on the left in Figure \ref{newEx} and let $B = B\upc{0} = \{1,2\}$. Consider the set of PSD forces $\calf = \{1 \rightarrow 7, 2 \rightarrow 5, 2 \rightarrow 3, 5 \rightarrow 6, 1 \rightarrow 4\}$ of $B$. In the first time step of $\calf$, $G$ breaks into two components with vertices $W_1 = \{5,6,7\}$ and $W_2 = \{3,4\}$ respectively and the forces $1 \rightarrow 7, 2 \rightarrow 5$, and $2 \rightarrow 3$ are performed simultaneously. In the second time step of $\calf$, $W_1$ and $W_2$ each break into one component with vertices $W_3 = \{6\}$ and $W_4 = \{4\}$ respectively. Then the forces $5 \rightarrow 6$ and $1 \rightarrow 4$ are performed simultaneously.
\end{ex}

\begin{figure}[H] \begin{center}
\scalebox{.85}{\includegraphics{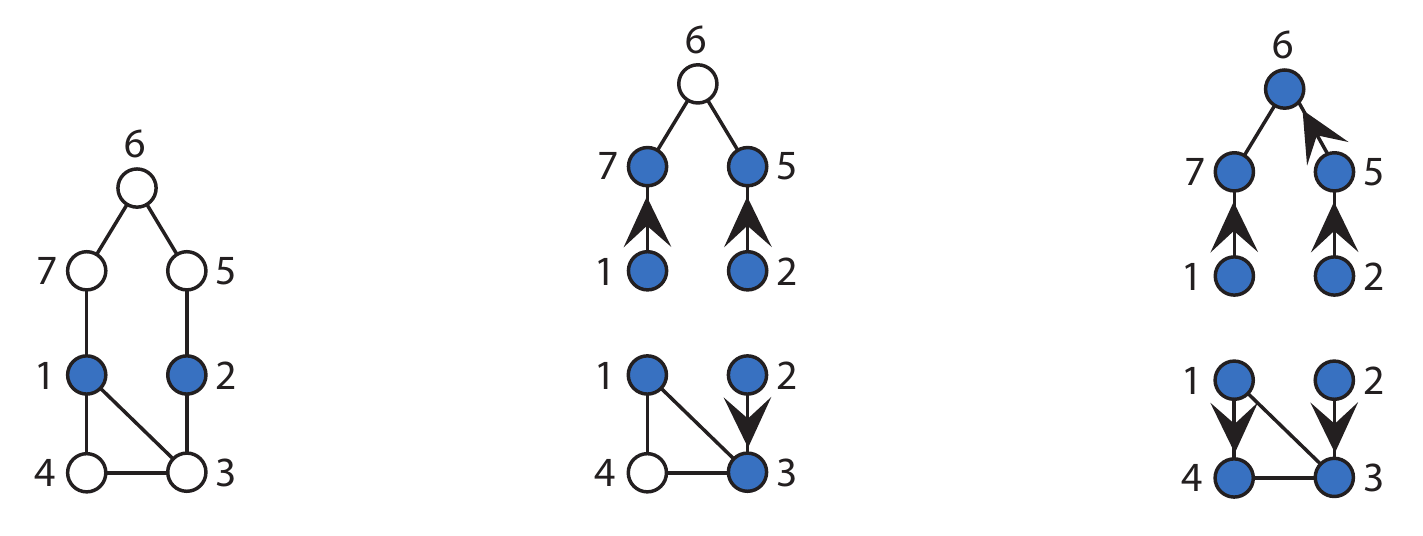}}\\
\end{center}
\hspace{1.05 in} $t=0$ \hspace{1.45 in} $t=1$ \hspace{1.55 in} $t=2$
\begin{center}
\caption{The PSD zero forcing process as seen from the reduction perspective.}\label{newEx} 
\end{center}
\end{figure}

Suppose $G$ is a graph, $B \subseteq V(G)$ is a PSD zero forcing set of $G$, and $\calf$ is a set of PSD forces of $B$ with $\ptp(G; \calf) = \ptp(G; B)$. Define $T(G; B; \calf)$ to be the rooted tree that represents the breakdown of components throughout the PSD reduction process where the edges of the tree are labeled by the components. Note that if two vertices $u$ and $v$ are equidistant from the root in $T(G; B; \calf)$, then $u$ can have a different number of children than $v$. For example, suppose  $G$ breaks into two components $W_1$ and $W_2$ in the first time step. In the second time step, suppose $W_1$ breaks into one component $W_3$ and suppose $W_2$ breaks into two components $W_4$ and $W_5$. In this case, $T(G; B; \calf)$ is the tree illustrated in Figure \ref{breakdownFig}.

\begin{figure}[H] \begin{center}
\scalebox{.85}{\includegraphics{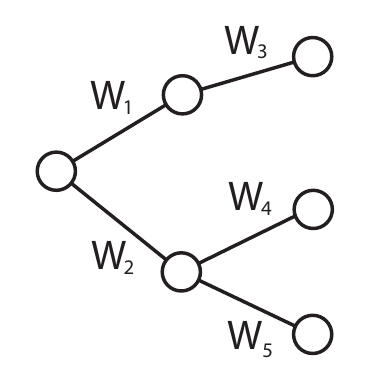}}\\
\caption{The component $W_2$ breaks into two components, but $W_1$ only breaks into one.}\label{breakdownFig} 
\end{center}
\end{figure}

For a graph $G$, PSD zero forcing set $B$, and set of PSD forces $\calf$, Definition \ref{forceTreeExtDef} (illustrated in Example \ref{forcingTreeExtEx}) uses the tree $T(G; B; \calf)$ to extend the forcing trees of $\calf$.

\begin{defn}\label{forceTreeExtDef}
For each $b \in B$, define $\mathcal{E}_b(\calf)$ to be the copy of $T(G; B; \calf)$ whose vertices are labeled as follows. 
\begin{enumerate}
\item Label the root of $T(G; B; \calf)$ as $b$.
\item Suppose $u$ is a vertex in the forcing tree $T_b(\calf)$ and $u$ becomes blue at time $t$ in component $W$. Label as $u$ the vertex in $T(G; B; \calf)$ that is distance $t$ from $b$ and is incident to the edge labeled $W$.
\item Give each remaining unlabeled vertex the label of its parent recursively.
\end{enumerate}
\end{defn}

\begin{ex}\label{forcingTreeExtEx}
Let $G$, $B$, and $\calf$ be given as in Example \ref{redPSDEx}. Then $\calf$ has two forcing trees $T_1(\calf)$ and $T_2(\calf)$. Note that $T_1(\calf)$ is a path on three vertices where vertices $7$ and $4$ are the two children of vertex $1$. Since vertex $1$ is the root of $T_1(\calf)$, the root of $\mathcal{E}_1(\calf)$ is labeled as $1$. Vertex $7$ is forced in the first time step of $\calf$ in component $W_1$ and vertex $4$ is forced in the second time step of $\calf$ in component $W_4$. The top row of Figure \ref{forcingTreeExtFig} illustrates the three steps in the construction of $\mathcal{E}_1(\calf)$. Likewise, the construction of $\mathcal{E}_2(\calf)$ is shown in the bottom row of Figure \ref{forcingTreeExtFig}.
\end{ex}

\begin{figure}[H] \begin{center}
\scalebox{.85}{\includegraphics{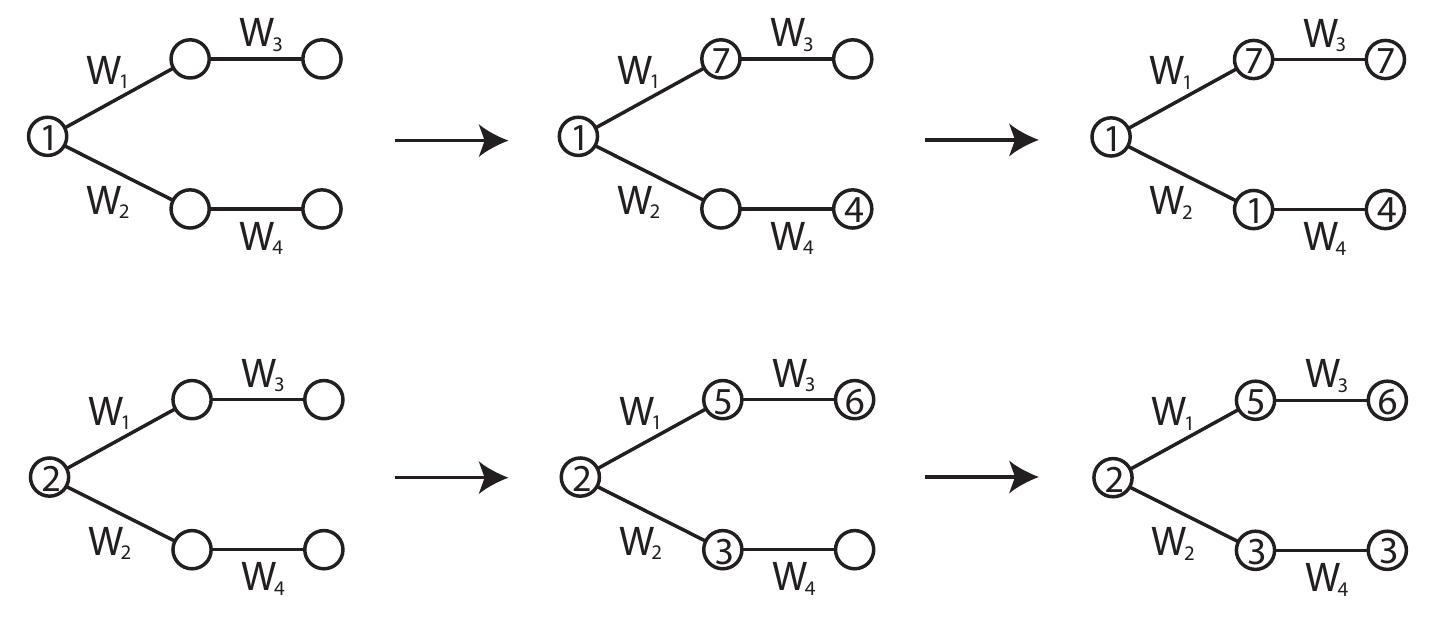}}\\
\caption{The construction of $\mathcal{E}_1(\calf)$ and $\mathcal{E}_2(\calf)$ is shown in the top and bottom row respectively.}\label{forcingTreeExtFig} 
\end{center}
\end{figure}

The next proposition concerns edges that are not contained in the forcing trees of a set of PSD forces. 
\begin{prop}\label{betweenEdgeProp}
Let $G$ be a graph with PSD zero forcing set $B \subseteq V(G)$. Suppose $\calf$ is a set of PSD forces of $B$ and $uv \in E(G)$ is not in any forcing tree of $\calf$. Choose $u',v' \in B$ such that $u \in T_{u'}(\calf)$ and $v \in T_{v'}(\calf)$. Then there exists a sequence of edge labels $W_1, W_2, \ldots, W_j$ such that starting at the root of $\mathcal{E}_{u'}(\calf)$ (respectively $\mathcal{E}_{v'}(\calf)$) and following the edges labeled $W_1, W_2, \ldots, W_j$ leads to a copy of vertex $u$ (respectively $v$).
\end{prop}

\begin{proof}
Suppose $u$ becomes blue at time $i$ and $v$ becomes blue at time $j$ with $i \leq j$. Let $W_1$, $W_2$, \ldots , $W_j$ be the sequence of components that contain $v$ during the first $j$ time steps of $\calf$. Therefore, the path in $\mathcal{E}_{v'}(\calf)$ obtained by starting at the root $v'$ and following the edges labeled $W_1$, $W_2$, \ldots, $W_j$ leads to a vertex labeled $v$. Since $uv \in E(G)$, $u$ is in components $W_1$, $W_2$, \ldots $W_i$ and $u$ cannot force in a future component contained in $W_i$ until $v$ becomes blue in time step $j$. Once $u$ becomes blue in component $W_i$, $u$ remains in the set of blue vertices that are attached to every future component that is contained in $W_i$. So $u$ is a blue vertex in the graph in which $v$ is forced in time step $j$. Thus, the path in $\mathcal{E}_{u'}(\calf)$ obtained by starting at the root $u'$ and following the edges labeled $W_1, W_2, \ldots, W_j$ leads to a copy of $u$. 
\end{proof}

Proposition \ref{betweenEdgeProp} is used in the following definition which extends a given graph $G$ to a major of $G$ using a set of PSD forces. 
\begin{defn} \label{newFullExtDef}
Suppose $G$ is a graph and $\calf$ is a set of PSD forces of a PSD zero forcing set $B \subseteq V(G)$ such that $\ptp(G;\calf) = \ptp(G; B)$. For each edge $uv \in E(G)$, let $t(uv)$ denote the earliest time in $\calf$ at which both $u$ and $v$ are blue. For each $v \in V(G)$, let $r(v)$ be the unique vertex in $B$ such that $v \in T_{r(v)}(\calf)$. Define the \emph{(PSD) extension of $G$ with respect to $B$ and $\calf$}, denoted $\mathcal{E}_+(G;B;\calf)$, to be the graph obtained by the following procedure.
\begin{enumerate}
\item Construct the graph $G_1= \dot{\bigcup}\{\mathcal{E}_b(\calf) \ | \ b \in B\}$.
\item For each edge $v_1v_2 \in E(G)$ with $v_1, v_2 \in B$, add to $G_1$ the edge that connects the root of $\mathcal{E}_{v_1}(\calf)$ to the root of $\mathcal{E}_{v_2}(\calf)$. Call the resulting graph $G_2$.

\item For each edge $v_1v_2 \in E(G)$ with $\{v_1,v_2\} \nsubseteq B$ that is not in any forcing tree of $\calf$, add to $T_2$ the edge that connects the copies of $v_1$ and $v_2$ that are distance $t(v_1v_2)$ away from the roots in $\mathcal{E}_{r(v_1)}(\calf)$ and $\mathcal{E}_{r(v_2)}(\calf)$ respectively.
\end{enumerate}
\end{defn}

 Note that 3. in Definition \ref{newFullExtDef} is possible by Proposition \ref{betweenEdgeProp}.
 
\begin{ex} \label{newFullExtEx}
Let $G$, $B$, and $\calf$ be the graph, PSD zero forcing set, and set of PSD forces given in Example \ref{redPSDEx}. Note that the edge $\{3, 4\}$ is not in either of the forcing trees of $\calf$. Vertex $4$ becomes blue after vertex $3$ and vertex $4$ is contained consecutively in the components $W_2$ and $W_4$. Therefore, there is an edge in $\mathcal{E}_{+}(G; B; \calf)$ that connects the vertices obtained by starting at the roots of $\mathcal{E}_1(\calf)$ and $\mathcal{E}_2(\calf)$ and following the edges $W_2$ and $W_4$. The graph $G$ is shown alongside the extension $\mathcal{E}_{+}(G; B; \calf)$ in Figure \ref{newFullExFig}.
\end{ex}

\begin{figure}[H] \begin{center}
\scalebox{.85}{\includegraphics{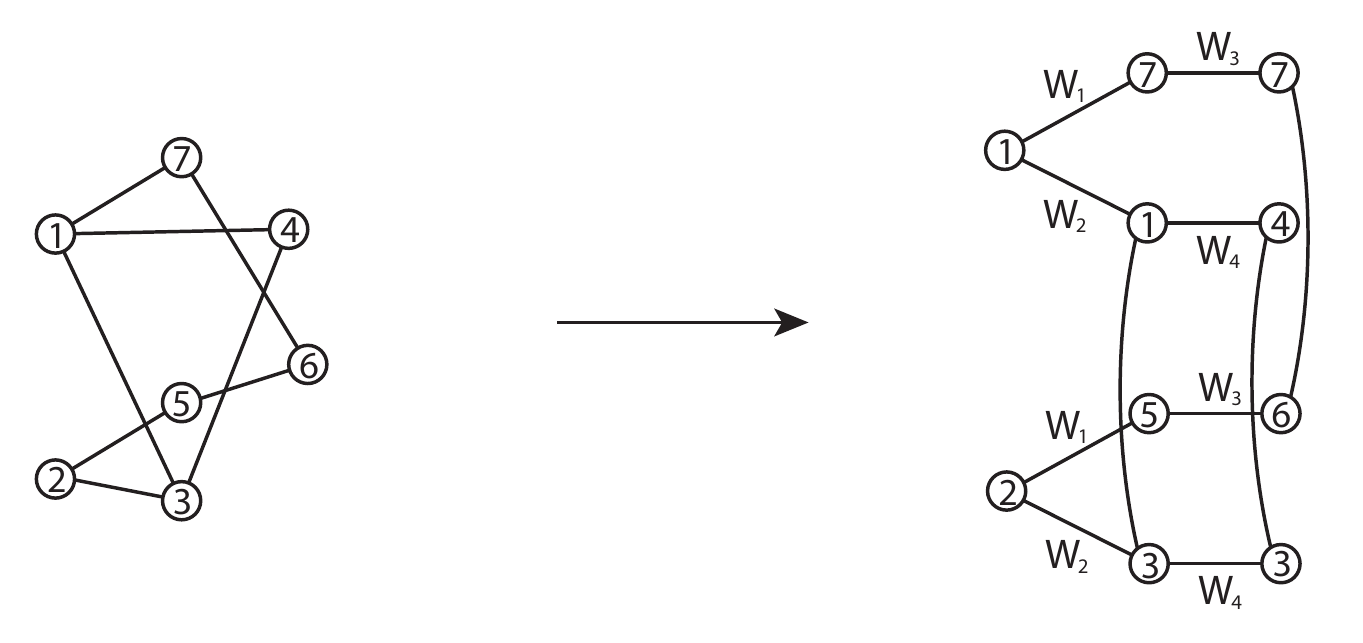}}\\
\end{center}
\begin{center}
\vspace{-.028 in}
\caption{A graph $G$ (left) is shown before and after its extension $\mathcal{E}_{+}(G; B; \calf)$ (right).}\label{newFullExFig} 
\end{center}
\end{figure}

\begin{rem}
In the case where each forcing tree of $\calf$ is a path, $\calf$ is a set of standard forces. In this case, $\mathcal{E}_+(G, B, \calf)$ is equal to the extension $\mathcal{E}(G, B, \calf)$ that is defined for standard zero forcing in \cite[Definition 3.12]{JCThrot}. Thus, Definition \ref{newFullExtDef} generalizes the extension given in \cite{JCThrot} to PSD zero forcing.
\end{rem}

\begin{lem}\label{treeContracLem}
If $G$ is a graph, $B \subseteq V(G)$ is a PSD zero forcing set of $G$, and $\calf$ is a set of PSD forces of $B$ with $\ptp(G; \calf) = \ptp(G; B)$, then contracting an edge in a forcing tree of $\calf$ does not increase the PSD propagation time of $\calf$.
\end{lem}

\begin{proof}
Consider induction on $\ptp(G;\calf)$. If $\ptp(G; \calf) = 0$, then there are no edges in any forcing tree of $\calf$ and Lemma \ref{treeContracLem} is vacuously true. Assume Lemma \ref{treeContracLem} holds for any $G'$ and $\calf'$ with $0 \leq \ptp(G'; \calf') \leq t-1$ and suppose that $G$ and $\calf$ satisfy $\ptp(G; \calf) = t$. It is shown in the proof of \cite[Lemma 3.15]{JCThrot} that in standard zero forcing, a vertex $v$ that is forced in the last time step can only be adjacent to the vertex that forced $v$ and vertices that do not perform a force. Therefore, if $v \in V(G)$ is forced during time step $t$ in component $W$, then $v$ can only be adjacent to the vertex that forced $v$ and other vertices in $W$ that are leaves of a forcing tree of $\calf$. So if $e$ is an edge that is used to perform a PSD force in $\calf$ during time step $t$, then contracting $e$ does not increase the PSD propagation time of $\calf$.

Now suppose $e = uv$ is an edge such that $u \rightarrow v$ in time step $i$ of $\calf$ for some $i < t$. Label the vertices of $G$ as $v_1, v_2, \ldots, v_{|G|}$ and let $G/e$ be the graph obtained from $G$ by contracting $e$ and labeling as $v$ the new vertex that is formed as a result of the contraction. Let $S$ be the set of vertices in $G$ that are forced last in $\calf$. Obtain the graph $G/e$ as follows. First, delete the vertices in $S$ from $G$. Next, contract the edge $e$. Finally, add the vertices in $S$ back to the graph preserving the original neighborhood of each vertex in $S$. Note that $\ptp(G-S; \calf) \leq t-1$. So by the induction hypothesis, the PSD propagation time of $\calf$ after contracting $e$ is also at most $t-1$. The vertices in $S$ are added back to the graph at the end of the forcing trees and each vertex in $S$ will become blue simultaneously in the final time step. Therefore, $\ptp(G/e; \calf) \leq t-1 +1 = t$.
\end{proof}

Recall that the \emph{depth} of a vertex $v$ in a rooted tree $T$ is the distance from $v$ to the root and the height of $T$ is the maximum depth of the vertices in $T$. For integers $k > 0$ and $b \geq 0$, let $T_{k,b}$ denote the rooted tree of height $b$ such that every vertex of depth less than $b$ has $k$ children. If $G$ is a graph of the form $K_a \square T_{k,b}$, define the \emph{tree edges} of $G$ to be the edges in each copy of $T_{k,b}$ in the Cartesian product. Likewise, define the \emph{complete edges} of $G$ to be the edges in each copy of $K_a$ in the Cartesian product. Similar to standard throttling, the extension in Definition \ref{newFullExtDef} can be used to give a structural characterization of graphs with a given PSD throttling number.

\begin{thm}\label{psdCharThm}
Suppose $G$ is a graph and $t$ is a fixed positive integer. Then $\thp(G) \leq t$ if and only if there exists integers $a,k > 0$ and $b \geq 0$ such that $a + b = t$ and $G$ can be obtained from $K_a \square T_{k,b}$ by contracting tree edges and/or deleting complete edges.
\end{thm}

\begin{proof}
Suppose $\thp(G) = t' \leq t$. Let $\calf$ be a set of PSD forces of a PSD zero forcing set $B\subseteq V(G)$ with $\ptp(G; \calf) = \ptp(G; B) = b'$. Choose $a = |B|$ and let $k$ be the maximum number of components in any time step of the PSD reduction process of $\calf$. Then $\mathcal{E}_+(G;B;\calf)$ can be obtained from $K_{a} \square T_{k,b'}$ by contracting tree edges and/or deleting complete edges. Note that $G$ can be obtained from $\mathcal{E}_+(G;B;\calf)$ by contracting the tree edges whose endpoints have the same label. Finally, if $b = t - a$, then $K_{a} \square T_{k,b'}$ can be obtained from $K_{a} \square T_{k,b}$ by contracting tree edges. Note that $a + b = t$.

Now suppose $G$ can be obtained from $K_a \square T_{k,b}$ by contracting tree edges and/or deleting complete edges. Let $B$ be the vertices in the copy of $K_a$ that corresponds to the root of $T_{k,b}$. Choose $\calf$ to be the set of PSD forces of $B$ obtained by having each vertex in every copy of $T_{k,b}$ force each of its children in that copy. Note that $\ptp(G;\calf) = b$ because no vertex is required to wait for multiple time steps in order to perform a force. This means that $\thp(K_a \square T_{k,b}) \leq a + b$. The tree edges of $K_a \square T_{k,b}$ are exactly the edges used in the forcing trees of $\calf$. By Lemma \ref{treeContracLem}, contracting these edges does not increase the PSD propagation time of $\calf$. Since the complete edges of $K_a \square T_{k,b}$ are not in any forcing tree of $\calf$, deleting these edges does not increase the PSD propagation time of $\calf$. Thus, if $G$ is obtained from $K_a \square T_{k,b}$ by contracting tree edges and/or deleting complete edges, then $\thp(G) \leq a +b$.
\end{proof}

It is shown in \cite{PSD} that if $T'$ and $T$ are trees with $T' \leq T$, then $\thp(T') \leq \thp(T)$ (i.e., the PSD throttling number is subtree monotone). This result can be extended to minors of trees as an immediate consequence of Theorem \ref{psdCharThm}.

\begin{cor}
If $T'$ and $T$ are trees with $T' \preceq T$, then $\thp(T') \leq \thp(T)$.
\end{cor}

In Section \ref{zpFloorThrot}, Theorem \ref{psdCharThm} is used to quickly obtain a similar characterization for a variant of PSD throttling.
\end{section}

\begin{section}{Throttling the Minor Monotone Floor of PSD Zero Forcing} \label{zpFloorThrot}

This section considers throttling for a variant of PSD zero forcing that allows hopping in each component. Let $G$ be a graph with $B \subseteq V(G)$ colored blue and $V(G) \setminus B$ colored white. Let $W_1, W_2, \ldots, W_k$ be the sets of white vertices in each connected component of $G-B$. For each $1 \leq i \leq k$, let $A_i \subseteq B$ be the set of vertices that are considered ``active" with respect to $W_i$. The \emph{$\zpf$ color change rule} is that if $u \in A_i$, $w \in W_i$, and every neighbor of $u$ in $G[W_i \cup B] - w$ is blue, then $u$ can force $w$ to become blue. (Note that if $w$ is the only white neighbor of $u$ in $G[B \cup W_i]$, then $u \rightarrow w$ is a $\zp$ force. Otherwise, $u$ has no white neighbors in $G[B \cup W_i]$ and $u \rightarrow w$ by hopping.) After $u \rightarrow w$, $u$ is removed from $A_i$ and $w$ becomes active with respect to $W_i$. 

It is shown in \cite{Parameters} that the minor monotone floor of $\zp$ of a graph $G$ (denoted $\zpf(G)$) can be defined as the $R$ forcing parameter, $R(G)$, where $R$ is the $\zpf$ color change rule. This allows for the study of $\zpf$ propagation time and $\zpf$ throttling. Since every PSD zero forcing set $B$ of a graph $G$ is also a $\zpf$ forcing set of $G$ with $\ptzpf(G; B) \leq \ptp(G; B)$, $\thzpf(G) \leq \thp(G)$. In \cite[Corollary 3.6]{JCThrot}, it is shown that for a graph $G$ and subset $B \subseteq V(G)$, $\thzpf(G;B) = \min\{\thp(H;B)\}$ where $H$ ranges over all spanning supergraphs of $G$. This leads to an analogous fact for the $\zpf$ throttling number of a graph.

\begin{cor}\label{PSDFspansupsCor}
If $G$ is a graph, then $\thzpf(G) = \min\{\thp(H) \ | \ G \leq H \text{ and } |G| = |H|\}$.
\end{cor}

\begin{proof}
Choose a subset $B \subseteq V(G)$ and a set $\calf$ of $\zpf$ forces of $B$ such that $\ptzpf(G; \calf) = \ptzpf(G; B)$ and $\thzpf(G) = \thzpf(G;B)$. Then 
\beas
\min\{\thp(H) \ | \ G &\le& H \text{ and } |G| = |H|\} \leq \min\{\thp(H; B) \ | \ G \leq H \text{ and } |G| = |H|\}\\[.5 em]
 &=& \thzpf(G;B) = \thzpf(G).
\eeas 
Let $H'$ be a spanning supergraph of $G$ such that $\thp(H') \leq \thp(H)$ for any spanning supergraph $H$ of $G$. Suppose $B' \subseteq V(H')$ with $\thp(H') = \thp(H'; B')$. Now suppose $\calf'$ is a set of PSD forces of $B'$ such that $\ptp(H'; B') = \ptp(H'; \calf')$. The next step is to show that $\calf'$ is a set of $\zpf$ forces of $B'$ in $G$ with $\ptzpf(G; \calf') \leq \ptp(H'; \calf')$. Choose an edge $uw \in E(H') \setminus E(G)$ and suppose $(u \rightarrow w) \in \calf'$. In the component where $u \rightarrow w$, $w$ is the only white neighbor of $u$. So if the edge $uw$ is removed from $E(G)$, $u$ is allowed to force $w$ by a hop. If $(u \rightarrow w) \notin \calf'$, then removing $uw$ does not slow down the propagation time of $\calf'$. Note that removing edges from $H'$ may increase the number of components at each time step when the blue vertices are removed. However, due to hopping, every force in $\calf'$ is still a valid $\zpf$ force in $G$ and $\ptzpf(G; \calf') \leq \ptp(H'; \calf')$. Thus,
\beas
\thzpf(G) &\leq& \thzpf(G; B') \leq \thzpf(G; \calf') \leq \thp(H'; \calf')\\[.5 em]
&=& \thp(H') = \min\{\thp(H) \ | \ G \leq H \text{ and } |G| = |H|\}.\qedhere
\eeas 
\end{proof}

Theorem \ref{psdCharThm} and Corollary \ref{PSDFspansupsCor} can be used to characterize graphs $G$ with $\thzpf(G) \leq t$ for any positive integer $t$. This characterization is also in terms of specified minors of the Cartesian product of a tree and a complete graph.

\begin{thm}\label{psdFloorCharThm}
Suppose $G$ is a graph and $t$ is a fixed positive integer. Then $\thzpf(G) \leq t$ if and only if there exists positive integers $a,k$ and non-negative integer $b$ such that $a + b = t$ and $G$ can be obtained from $K_a \square T_{k,b}$ by contracting tree edges and/or deleting edges.
\end{thm}

\begin{proof}
Suppose $\thzpf(G) \leq t$. By Corollary \ref{PSDFspansupsCor}, there exists a spanning supergraph $H$ of $G$ such that $\thp(H) = \thzpf(G) \leq t$. Clearly $G$ can be obtained from $H$ by removing edges. By Theorem \ref{psdCharThm}, there exists positive integers $a,k$ and non-negative integer $b$ such that $a + b = t$ and $H$ can be obtained from $K_a \square T_{k,b}$ by contracting tree edges and/or deleting complete edges. Thus, $G$ can be obtained from $K_a \square T_{k,b}$ by contracting tree edges and/or deleting edges.

Let $T = K_a \square T_{k,b}$ for some positive integers $a,k$ and non-negative integer $b$. Suppose $D \subseteq E(T)$ and $C$ is a set of tree edges of $T$ such that $C \cap D = \emptyset$ and $G$ can be obtained from $T$ by contracting the edges in $C$ and deleting the edges in $D$. Let $T'$ be the graph obtained from $T$ by contracting the tree edges in $C$. By Theorem \ref{psdCharThm}, $\thp(T') \leq a + b$. Note that $G$ can be obtained from $T'$ by deleting the edges in $D$. By Corollary \ref{PSDFspansupsCor}, $\thzpf(G) \leq \thp(T') \leq a + b$.
\end{proof}

Although Theorems \ref{psdCharThm} and \ref{psdFloorCharThm} can be useful in describing graphs with low throttling numbers, it less helpful when considering graphs with large throttling numbers relative to the number of vertices. The next section characterizes graphs with high throttling numbers using families of forbidden subgraphs.

\end{section}

\begin{section}{Throttling as a Forbidden Subgraph Problem}\label{sectFSP}

 In this section, we consider graphs $G$ with standard and PSD throttling numbers at least $|V(G)| - k$ for some integer $0 \leq k \leq |V(G)|$. Following convention, we drop the subscript in the notation for standard throttling and propagation. Given an initial set $B \subseteq V(G)$ of blue vertices, recall that for standard or PSD zero forcing, $B^{[t]}$ denotes the set of blue vertices in $G$ at time $t$ and $B^{(t)}$ denotes set of vertices that turn blue at time step $t$ given a forcing process $F$. Notice that if the standard (or PSD) propagation time of $B$ in $G$ is $r$, then $B\up{i}$ is non-empty for each $0\leq i \leq r$ and $\bigcup_{i=0}^rB^{(i)}=V(G)$. Let $U^{(t)}$ be the set of blue vertices in $G$ that force the vertices in $B^{(t)}$ to become blue at time $t$ using a set of forces $\calf$. A set $B$ is a \emph{witness of $\throt(G) < n-k$} if $\throt(G;B) < n-k$.
 
 The first characterization of throttling numbers in terms of forbidden subgraphs that we are aware of appears in \cite{PSD}. 
 
 \begin{figure}[H] 
\label{n-1fig1}
\begin{center}
\scalebox{.55}{\includegraphics{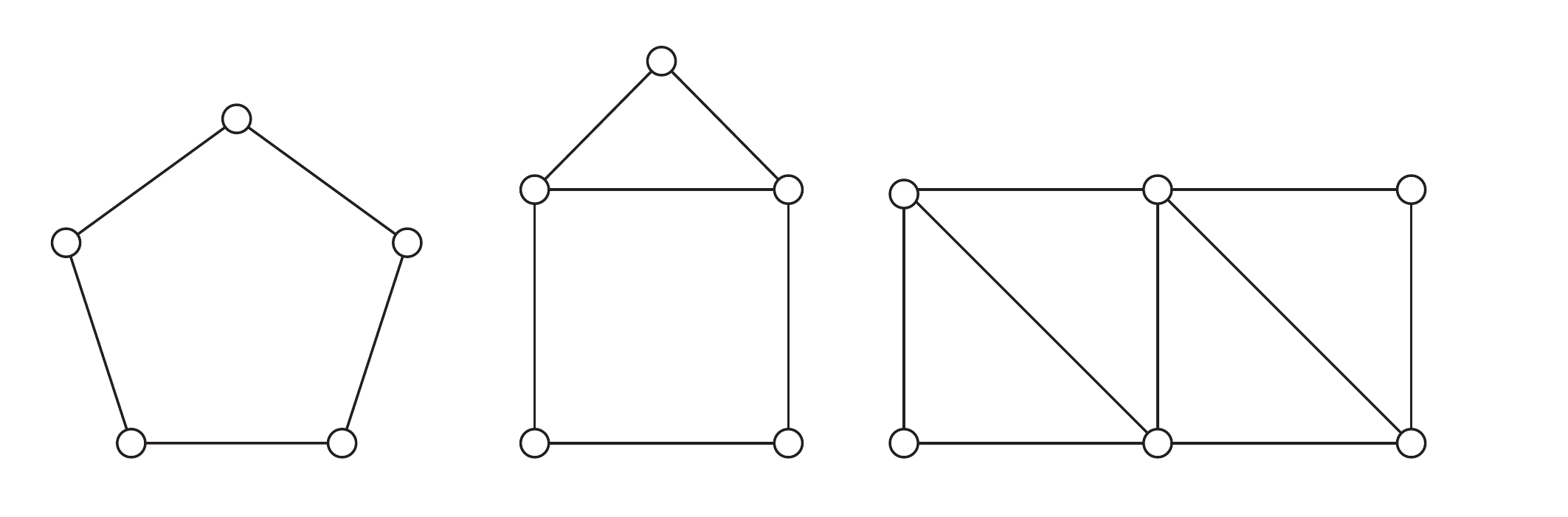}}
\caption{The $C_5$, house, and double diamond graphs are shown left to right.}
\end{center}
\end{figure}

\begin{thm}[\cite{PSD}]\label{th+n-1}
For a connected graph $G$, $\thp(G)\geq |V(G)|-1$ if and only if $G$ does not have an induced $\bar K_3$, $C_5$, house graph, or double diamond graph. Furthermore, $\thp(G)=|V(G)|-1$ if and only if $G$ does not have an induced $C_5$, house graph, or double diamond graph but $G$ has an induced $\bar K_2$. Finally, 
$\thp(G)=|V(G)|$ if and only if $G$ does not contain an induced $\bar K_2$ (that is, $G$ is complete).
\end{thm}

By Theorem \ref{th+n-1} it is clear that the sets of forbidden subgraphs that characterize $\thp(G)\geq |V(G)|$ and $\thp(G)\geq |V(G)|-1$ are finite. We derive a similar result to Theorem \ref{th+n-1} for standard zero forcing.

 \begin{figure}[H] 
\label{highpsdfig}
\begin{center}
\scalebox{.55}{\includegraphics{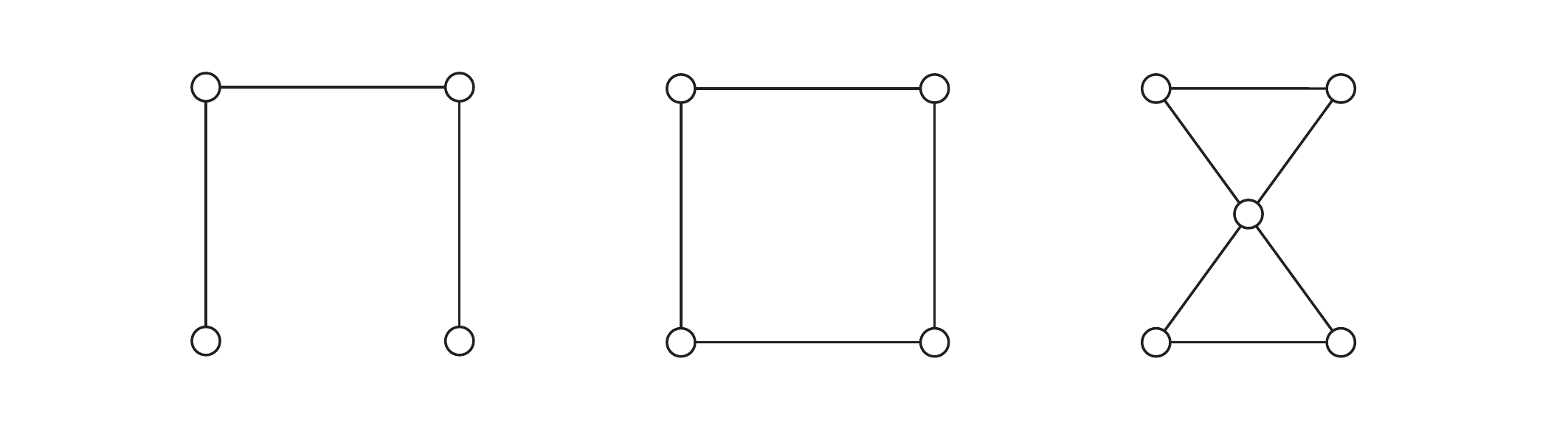}}
\caption{The $P_4$, $C_4$, and bowtie graphs are show left to right.}
\end{center}
\end{figure}

\begin{thm}\label{th=n}
For a connected graph $G$, $\throt(G)=|V(G)|$  if and only if $G$ does not contain an induced $P_4,C_4,$ or bowtie graph.
\end{thm}

\begin{proof}Let $G$ be a graph on $n$ vertices.
Notice that $\thr G\leq n$ for all $G$. Suppose that $\thr G<n$. This implies that there exists a time $t$ such that $|B^{(t)}|\geq 2$. Let $u,v\in B^{(t)}$ and choose $x,y\in U^{(t)}$ such that $x\rightarrow u$ and $y\rightarrow v$. By the standard color change rule, we have that $xu,yv\in E(G)$ and $xv,yu\notin E(G)$. There are two cases, either $xy \in E(G)$ or $xy\notin E(G)$. 

\textbf{Case 1:} Assume that $xy\in E(G)$. If $uv\notin E(G)$, then $uxyv$ induces a $P_4$. If $uv\in E(G)$, then $uxyv$ induces a $C_4$.

\textbf{Case 2:} Assume that $xy\notin E(G)$. Since $G$ is connected there exists a shortest path $P$ from $u$ to $v$. Notice that if $P=uv$, then $xuvy$ induces a $P_4$. If $P$ contains at least $4$ vertices, then $P$ induces a $P_4$. Therefore, $P=vzu$ for some vertex $z$. Finally, this implies that $xz,yz\in E(G)$, otherwise there is an induced $P_4$. Now, $xuzvy$ induces a bowtie graph.

In all cases, we have determined that if $\thr G<n$, then $G$ contains a $P_4,C_4,$ or bowtie graph.

To prove the converse, suppose that $G$ has an induced $P_4,C_4,$ or bowtie graph. In any case, there exists a  matching with edges $xy,uv$ such that $xv,uy\notin E(G)$. Let $B=V(G)\setminus \{y,v\}$. Now $x$ can force $y$ and $u$ can force $v$ in the first time step of the zero forcing process. Since $y$ and $v$ are the only white vertices in $G$ at time $0$, we can conclude that 
\[\throt(G)\leq |B|+\pt(G;B)=n-1.\qedhere\] 
\end{proof}

These theorems suggest that throttling for standard and PSD zero forcing can be treated as a forbidden induced subgraph problem. Proposition \ref{forbidden} confirms this suspicion.

\begin{prop}\label{forbidden}
Let $k$ be a constant. The set of graphs $G$ such that  $\throt(G)\geq |V(G)|-k$  and $|V(G)|\geq k$ is characterized by a family of forbidden induced subgraphs. Similarly, the set of graphs $G$ such that   $\thp(G)\geq |V(G)|-k$ and $|V(G)|\geq k$ is characterized by a family of forbidden induced subgraphs.
\end{prop}

\begin{proof}
 Suppose that $\thr G<|V(G)|-k$ and let $H$ be any graph such that $G$ is an induced subgraph of $H$ with the injection $\phi:V(G)\to V(H)$.
 Let $B\subseteq V(G)$ be a zero forcing set that realizes $\throt(G; B) = \thr G < V(G)-k$ and let $W = V(G) \setminus B$. 
 Then $B'=V(H)\setminus \phi(W)$ is a zero forcing set of $H$.  
 This follows from the fact that if $v\rightarrow u$ is possible in $G$ given $B$, then $\phi(v)\rightarrow \phi(u)$ is possible in $H$ given $B'$.
 In particular, \[\throt(H)\leq |B'|+\pt(H;B')=|V(H)\setminus \phi(V(G))|+|B|+\pt(G;B)<|V(H)|-k.\] Therefore, $B'$ is a zero forcing set of $H$ that demonstrates that $\throt (H)<V(H)-k$.  The proof for PSD throttling is the same.
\end{proof}
Let $\mathcal G_k$ be a set of forbidden graphs that characterizes graphs $G$ with $\throt(G)\geq |V(G)|-k$.
A natural question to ask given Proposition \ref{forbidden} is how large $\mathcal G_k$ must be to characterize graphs $G$ with $\throt(G)\geq |V(G)|-k$.  In order to show that $\mathcal G_k$ can be finite, we introduce the idea of ``savings". Throttling is an optimization between how many vertices are chosen in the initial zero forcing set and its propagation time. Intuitively, we want to force multiple vertices in a single time step to optimize throttling. In these cases, we ``save'' ourselves from choosing vertices in the initial zero forcing set by efficiently forcing vertices during the zero forcing process. To capture this idea, we consider the quantity $|B^{(t)}|-1$, which represents how much we ``save" at time $t$. This quantity is the number of efficiently forced vertices at time $t$ at the cost of waiting a time step. The following lemma states that in order to reduce the throttling number, we must efficiently force vertices.

\begin{lem}\label{savings}
Let $G$ be a graph and suppose $R$ is either the standard or PSD color change rule. Then, $\throt_R(G) < |V(G)|-k$ if and only if there exists an $R$ forcing set $B \subseteq V(G)$ such that 
 \[\sum_{i=1}^{\pt_R(G;B)} |B^{(i)}|-1  \geq k+1.\] 
\end{lem}

\begin{proof}
Let $B$ be a standard zero forcing set of $G$ with \[\sum_{i=1}^{\pt(G;B)} |B^{(i)}|-1  \geq k+1.\] This implies that 
\beas
    |V(G)\setminus B|-\pt(G;B)&\geq& k+1\\[.5 em]
    |V(G)|-|B|-\pt(G;B)&\geq& k+1\\[.5 em]
    |V(G)|-k-1&\geq& |B|+\pt(G;B)\\[.5 em]
    |V(G)|-k&>& \throt(G).
\eeas

To prove the converse, assume that $|V(G)|-k>\throt(G)$ and let $B$ be a zero forcing set that realizes this inequality. In particular, suppose that
\[|V(G)|-k-1\geq |B|+\pt(G;B).\]
This implies that 
\[|V(G)\setminus B|-\pt(G;B)\geq k+1.\]
 Since $B$ is a zero forcing set, we can partition $V(G)\setminus B$ into $B^{(i)}$ for $1\leq i \leq \pt(G;B)$. Using this partition, we can count the elements in $V(G)\setminus B$ to obtain 
 \[\sum_{i=1}^{\pt(G;B)} |B^{(i)}|-1  \geq k+1.\]
 The proof for PSD zero forcing is exactly the same.
\end{proof}

Another important observation is that we can always choose a zero forcing set $B \subseteq V(G)$ such that $|B\up{i}| - 1 > 0$ at each time step $i$. In particular, suppose that $B$ is a zero forcing set of $G$ such that $\throt(G)\leq |B| + \pt(G; B) < |V(G)|-k$, and let \[A=\bigcup_{|B^{(i)}|=1} B^{(i)}.\] Then $B \cup A$ satisfies $\throt(G)\leq  |B \cup A| + \pt(G; B \cup A) < |V(G)|-k$.

\begin{defn}
We say a zero forcing set $B\subseteq V(G)$ is a \emph{standard witness for $\throt(G)< |V(G)|-k$}, if $|B\up i|-1>0$ for each time step $i$ and $|B| + \pt(G; B) < |V(G)|-k$. The same notion holds for PSD zero forcing. 
\end{defn}

\begin{ex}
Consider the graph $G$ given in Figure \ref{fig:standardwitness}. The initial zero forcing set (in blue) on the left has $|B\up 0|=3$, $|B\up 1|=2$, $|B\up 2|=1$, $|B\up 3|=1$, and $|B\up 4|=2$. Therefore, $B$ is not a standard witness for $\throt(G)< |V(G)|-1$. However,  $B'=B\up 0\cup B\up 2\cup B\up 3$, depicted (in blue) on the right of Figure \ref{fig:standardwitness}, is a standard witness for $\throt(G)<|V(G)|-1$.
\end{ex}

\begin{figure}[H] 

\begin{center}
\scalebox{.7}{\includegraphics{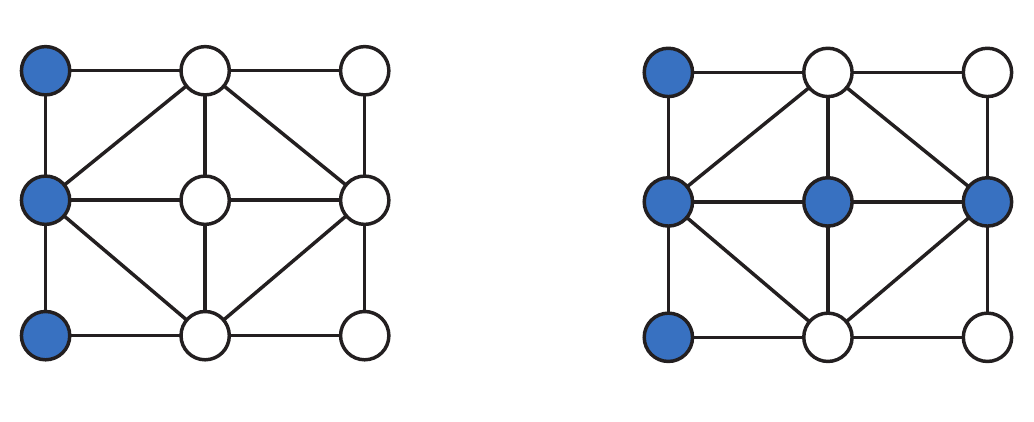}}
\caption{A graph $G$ with a witness and a standard witness for  $\throt(G) < |V(G)| - 1$ are shown on the left and right, respectively.}\label{fig:standardwitness}
\end{center}
\end{figure}

\vspace{-.2 in}

With these ideas, we can qualify what high throttling numbers mean in terms of the structure of a graph and how a zero forcing set behaves on it. In particular, Theorem \ref{finite} shows that there are finitely many graph structures that permit zero forcing (PSD zero forcing) sets to behave efficiently in their forcing behavior. Forbidding these structures ensure high throttling number.

\begin{thm}\label{finite}
Let $k$ be a non-negative integer and suppose $R$ is either the standard or PSD color change rule. The set of graphs $G$ such that $\throt_R(G)\geq |V(G)|-k$ and $|V(G)|\geq k$ is characterized by a finite family of forbidden induced subgraphs. 
\end{thm}

\begin{proof}
Let $k$ be a non-negative integer and $\mathcal G$ be the set of all graphs $G$ such that $\throt_R (G)\leq |V(G)|-k-1$ and $|V(G)|\leq 4k+4$. We will prove the claim that if $\throt_R(G)<|V(G)|-k$ and $|V(G)|\geq k$, then $G$ contains a graph in $\mathcal G$ as an induced subgraph. By Lemma \ref{savings}, there exists a zero forcing set $B$ such that 
\[\sum_{i=1}^{\pt_R(G;B)} |B^{(i)}|-1  \geq k+1.\] Without loss of generality, assume that $B$ is a standard witness for $\throt_R(G)< |V(G)|-k$.
Let $r$ be the first time step at which $\sum_{i=1}^{r} |B^{(i)}|-1  \geq k+1.$  In fact, we can choose $\hat B ^{(r)}\subseteq B^{(r)}$ so that \[|\hat B^{(r)}|-1+\sum_{i=1}^{r-1} |B^{(i)}|-1 =k+1.\] To avoid cumbersome notation, let $\hat B^{(i)}= B\up i$ for each $1\leq i \leq r-1$ so that \[\sum_{i=1}^{r}|\hat B\up{i}| - 1 = k+1.\] Since $B$ is a standard witness for $\throt_R(G) < |V(G)| - k$, $r\leq k+1$. Let $H=G[X]$ where \[X=\bigcup_{i=1}^r U^{(i)}\cup \hat B^{(i)}.\]

First, we will show that $\throt_R (H)\leq |V(H)|-k-1$. Then, we will show that $|V(H)|\leq 4k+4$. This will prove that $H$ is in $\mathcal G$.

Let \[\hat B=\bigcup_{i=1}^r \left(U^{(i)}\setminus \bigcup_{j=1}^{i-1}\hat  B^{(j)}\right).\] We will prove that $\hat B\up i$ is blue after time step $i$ by induction on $i$, assuming that $\hat B$ is the initial zero forcing set. As a base case, $\hat B$ is a set of blue vertices in $H$ after $0$ time steps by construction. We will assume that the sets $\hat B^{(j)}$ for $0\leq j\leq i-1$ are blue at the beginning of time step $i$. This implies that $U^{(i)}$ is blue at the beginning of time step $i$. Since $H$ is an induced subgraph of $G$ that contains $U^{(i)}$ and $\hat B^{(i)}$, the set $U^{(i)}$ can force $\hat B^{(i)}$ in $H$. Therefore, after time step $i$, the vertices in $\hat B^{(i)}$ are blue in $H$. Thus, $\hat B$ can force all of $H$ in at most $r$ time steps. Now, 
\[\throt_R(H)\leq |V(H)|-\sum_{i=1}^r |\hat B^{(i)}|-1= |V(H)|-k-1\]
by Lemma 4.4.

Notice that $|U^{(i)}|\leq |\hat B^{(i)}|$ by the standard color change rule (this is an equality for standard zero forcing, but can be an inequality for PSD zero forcing). Therefore,
\[ |X|\leq \sum_{i=1}^{r} |U^{(i)}|+|\hat B^{(i)}|\leq 2\sum_{i=1}^{r}|\hat B^{(i)}|= 2(k+1+r)\leq 4k+4.\] Thus, $H=G[X]$ is a graph in $\mathcal G$. 
\end{proof}

Notice that it is substantially easier to get a handle on the set of graphs $G$ such that $\throt(G)\leq |V(G)|-k-1$ and $|V(G)|\leq 4k+4$, than the infinite set of graphs $G$ with $\throt(G)\geq |V(G)|-k$ and $|V(G)|\geq k$. In particular  \cite[Theorem 4.1]{JCThrot} provides a characterization of graphs $G$ such that $\throt(G)\leq t$ for a positive integer $t$. Using this Theorem, there is a characterization of graphs $G$ with $\throt(G)\leq t:= 4k+4-k-1=3k+3$, which contains the set $\mathcal G$ used in the proof of Theorem \ref{finite}. Additionally, Theorem \ref{psdCharThm} is the PSD analog to \cite[Theorem 4.1]{JCThrot}, and can be used in the same way in relation to Theorem \ref{finite}.

Now that we have established that graphs $G$ with $\throt(G)\geq |V(G)|-k$ can be characterized by a finite set of forbidden subgraphs $\mathcal G_k$, we want to establish what these sets $\mathcal G_k$ can look like. To this end, consider the following definition. 

\begin{defn}
A graph $G$ is an
\emph{ $a$-accelerator} for integer $a\geq 1$ if $V(G)$ can be partitioned into sets $S$ and $T$, each of size $a+1$, such that there exists a matching between $S$ and $T$, and the only edges between $S$ and $T$ are in this matching.
\end{defn}

See Figure \ref{fig:accelerator} for an example of an $a$-accelerator.

\begin{figure}[H] 
\begin{center}
\scalebox{.7}{\includegraphics{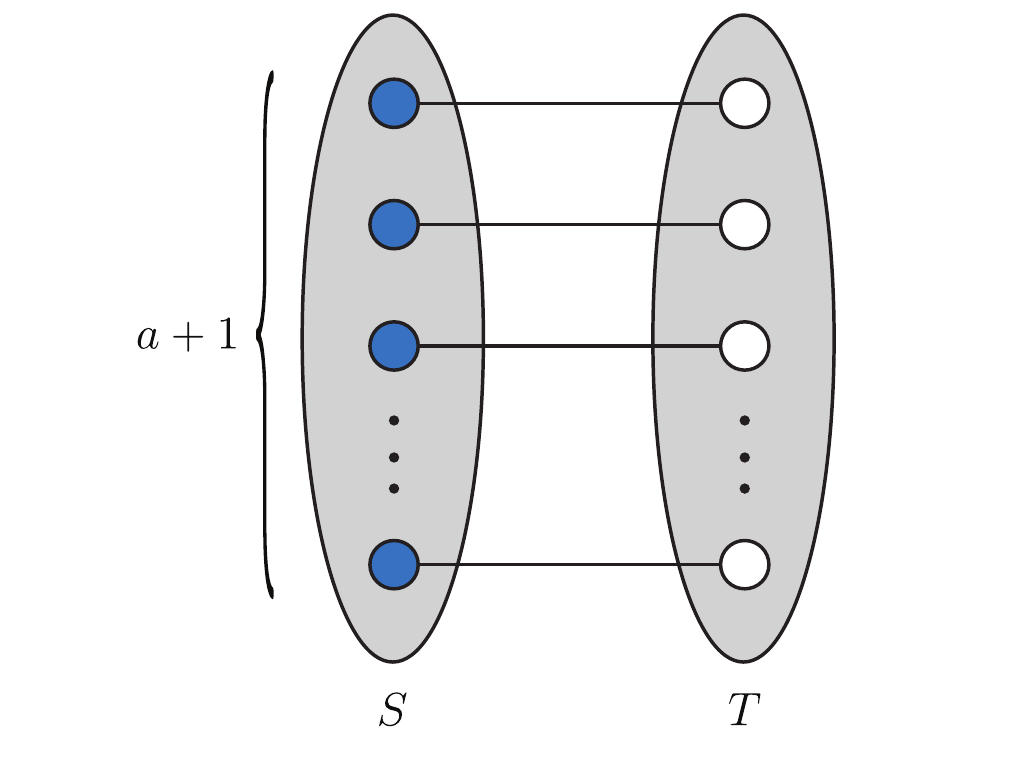}}
\caption{This is a general depiction of an $a$-accelerator. The grey areas may have any configuration of edges.}\label{fig:accelerator}
\end{center}
\end{figure}

Notice that if $G$ is an $a$-accelerator, then $\throt(G)\leq |V(G)|-a$ by using $S$ as a zero forcing set. Therefore, if $G$ is a $(k+1)$-accelerator, then  $\mathcal G_k$ contains an induced subgraph of $G$. Let $\calm_a$ be the set of $a$-accelerator graphs. To generalize accelerator graphs the following definition.

\begin{defn}Let $a_1,\dots, a_r$ be positive integers. A graph $G$ is an \emph{$(a_1,\dots,a_r)$-accelerator graph} if  $V(G)$ can be written as $\bigcup_{i=1}^r S_i\cup T_i$ where $S_i$ and $T_i$ are sets of $a_i+1$ vertices for $1\leq i\leq r$ such that:
\begin{enumerate}
    \item $\{S_i\}^r_{i=1}$ is a set of disjoint sets,
    \item $\{T_i\}^r_{i=1}$ is a set of disjoint sets, and
    \item $T_i\cap S_j$ is empty whenever $i+1\neq j $.
\end{enumerate}
Furthermore, the edges of $G$ must be partitioned by $S_i,T_i$ for $1\leq i\leq r$ such that:
\begin{enumerate}\setcounter{enumi}{3}
\item  $S_i$ and $T_i$ are perfectly matched (no non-matching edges exist between $S_i$ and $T_i$),
\item some edges are contained in $S_i$ or  $T_i$,
\item some edges go from $S_i\setminus T_{i-1}$ to $\bigcup_{j=1}^{i-1} S_j\cup T_j$, and
\item  $S_i$ is dominated by $T_{i-1}$ for $2\leq i\leq r$.
\end{enumerate}
\end{defn}

Notice that each $G[S_i\cup T_i]$ is an $a_i$-accelerator for each $i$ by properties 4 and 5. Furthermore, the edges in $G$ are restricted so that \[B=S_1\cup \bigcup_{i=2}^r S_i\setminus T_{i-1}\] is a standard witness for 
$\throt(G)< |V(G)|-(a_1+\cdots +a_r)+1$. In particular, $B\up i= T_i$ by property 7. Let $\calm_{a_1,\dots,a_r}$ denote the set of $(a_1,\dots,a_r)$-accelerator graphs. See Figure \ref{fig:joinedAccelerators} for an example of a $(2,3,2)$-accelerator.
\vspace{-.1 in}
\begin{figure}[H] 
\label{fig:joinedAccelerators}
\begin{center}
\scalebox{.65}{\includegraphics{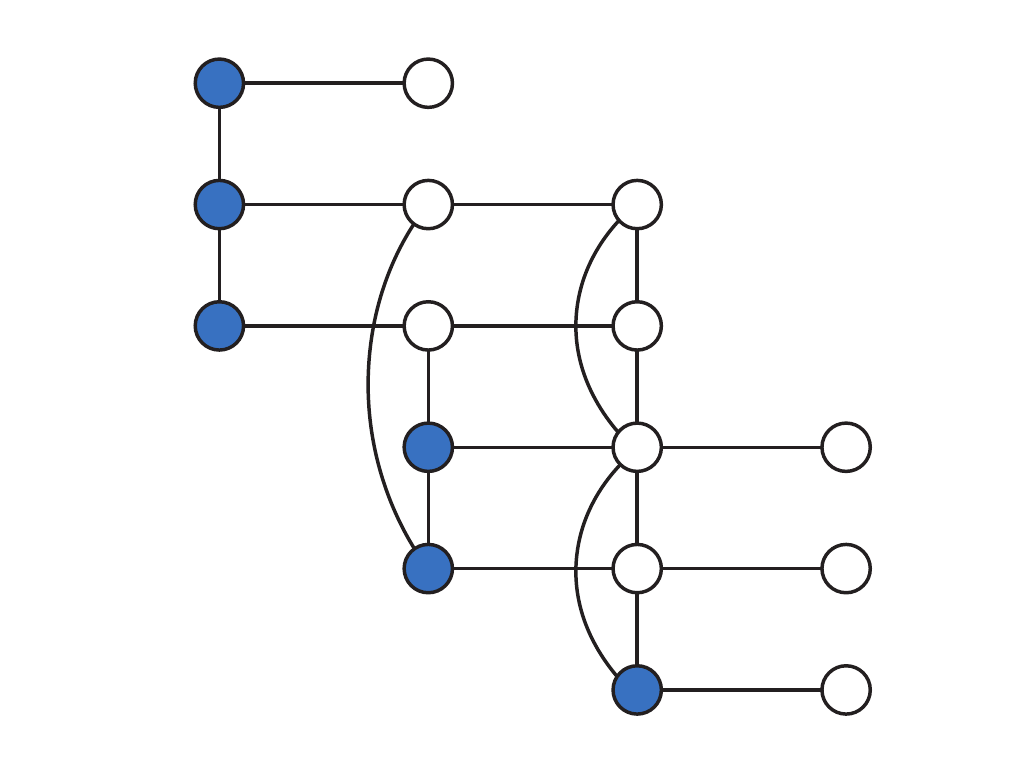}}\end{center}
\vspace{-.1 in}
\hspace{3 in} $M$
\caption{A graph $M\in \calm_{2,3,2}$.}
\end{figure}

\begin{obs}\label{trick}
If $M\in \calm_{a_1,\cdots,a_r}$, then \[\throt (M)\leq |V(M)|-(a_1+\cdots+a_r).\]
\end{obs}

With this observation, we can build a particular set of forbidden subgraphs $\G_k$.

\begin{thm}\label{accelerator}
Let \[\mathcal G_k=\bigcup_{a_1+\cdots+a_r=k+1} \calm_{a_1,\cdots,a_r}.\] Then $\throt(G)\geq |V(G)|-k$ if and only if $G$ does not contain a graph in $\mathcal G_k$ as an induced subgraph.
\end{thm}

\begin{proof}
By Observation \ref{trick}, if $G$ contains a graph in $\mathcal G_k$, then $\throt(G)< |V(G)|-k$. We will continue by assuming that $\throt(G)<|V(G)|-k$ and endeavor to find an induced subgraph of $G$ that is in $\mathcal G_k$. Define $H=G[X]$ where \[X=\bigcup_{i=1}^r U^{(i)}\cup \hat B^{(i)}\] as in the proof of Theorem \ref{finite} and let $a_i=|U\up i|-1$. We claim that $H \in \mathcal G_k.$ In particular, $U\up i, B\up i$ for $1\leq i\leq r$ is a decomposition of the vertices of $H$ such that $H$ is  an $(a_1,\dots,a_r)$-accelerator (where $ U\up i=S_i$ and $ B\up i= T_i$). First, notice that $U^{(i)}$ and $\hat B^{(i)}$ have the same cardinality $|U^{(i)}|=a_i+1$. Furthermore, $U\up i$ and $\hat B\up i$ are matched by a set of forces and the only edges from $U^{(i)}$ to $\hat B^{(i)}$ are in this matching. That is, \[H[U^{(i)}\cup \hat B^{(i)}]\in \calm_{a_i}.\]

Notice that vertices in $U^{(i+1)}\setminus B^{(i)}$ do not perform a force until time step $i+1$ by definition. Therefore, each vertex in $U^{(i+1)}\setminus B^{(i)}$ must be adjacent to a vertex in $B^{(i)}$. Thus, 
$H$ is an $(a_1,\dots, a_r)$-accelerator and $H\in \mathcal G_k.$
\end{proof}

Note that the set \[\mathcal G_k=\bigcup_{a_1+\cdots+a_r=k+1} \calm_{a_1,\cdots,a_r}\] is not a minimum set of forbidden subgraphs that characterize $\throt(G)\geq |V(G)|-k$. Consider the fact that $K_2\square P_4$ is contained in $\calm_{1,1,1}$, $\calm_{1,1}$, and $\calm_3.$ This means that $K_2\square P_4\in \mathcal G_1$ even though $\throt(K_2\square P_4)\leq 5$. The issue is that $K_2\square P_4$ contains an induced copy of $K_2\square P_3$, which is a graph in $\calm_2$ and, therefore, also a graph in $\mathcal G _1$.

Another instructive observation is that we can use the sets $\mathcal G_i$ to characterize graphs with $\throt(G)=|V(G)|-k$. In particular, we have the following corollary:

\begin{cor}\label{exact}
Let $G$ be a graph and \[\mathcal G_k=\bigcup_{a_1+\cdots+a_r=k+1} \calm_{a_1,\cdots,a_r}.\] Then $\throt(G)=|V(G)|-k$ if and only if $G$ does not contain any graph in $\mathcal G_k$ as an induced subgraph, but contains a graph in $\mathcal G_{k-1}$ as an induced subgraph. 
\end{cor}

In so far as we are concerned with standard zero forcing, the proof of Theorem \ref{accelerator} is a more detailed version of the proof of Theorem \ref{finite}. Theorem \ref{accelerator} highlights the role of accelerator graphs during the throttling process. In particular, given a large graph with relatively low throttling number ($n$ large, $k$ significant, and $\throt(G)=n-k$), we expect to see sufficiently many or sufficiently large accelerator graphs in the sense that $G$ must contain a subgraph in $\mathcal G_{k-1}$. However, $G$ cannot contain too many or too large accelerators in the sense that $G$ does not contain a graph in $\mathcal G_k$.  This is the moral captured in Corollary \ref{exact}. 

We can relate accelerator graphs to the construction of graphs with low standard throttling number in \cite{JCThrot}. In particular, \cite[Theorem 4.1]{JCThrot} provides a ``top down" construction of graphs with a specific throttling number. Accelerator graphs provide a ``bottom up" description of the structural properties of graphs with a specific throttling number. Together, these characterizations show the dual nature of high versus low throttling number. That is, understanding which graphs that  achieve $\throt(G)\geq |V(G)|-k$ is dual to understanding the graphs that achieve $\throt(G)\leq k$.

\end{section}

\begin{section}{Concluding Remarks}\label{conclusion}

To conclude our paper, we would like to make a few tangential remarks about our results that inform possible directions of future work. The first of these remarks concerns spectral graph theory, and is a path we stumbled upon by chance. Consider the following result:

\begin{cor}[Corollary 8.1.8 in \cite{CRS_2010}]\label{spectral_radius}
Let $G$ be a connected graph with maximal spectral radius among connected graphs on $n$ vertices with $m$ edges. Then $G$ does not contain $2K_2, P_4, C_4$ as an induced subgraph. 
\end{cor}

The family $R=\{2K_2, P_4,C_4\}$ is similar to the family we found in Theorem \ref{th=n} (also note that $\calm_1=\{2K_2, P_4,C_4\}$). In fact, $R$ can also be used as a family of forbidden subgraphs to characterize graphs $G$ with $\throt(G)=|V(G)|$ by Theorem \ref{accelerator}. With this in mind, we obtain a neat little result.

\begin{cor}\label{th_spectral}
If $G$ is a graph with maximal spectral radius for its adjacency matrix among connected graphs on $n$ vertices and $m$ edges, then $\throt(G)=n$.
\end{cor}

\begin{proof}
By Corollary \ref{spectral_radius}, $G$ is connected but does not contain $2K_2,P_4,C_4$ as an induced subgraph. Since $G$ does not contain $2K_2$ as an induced subgraph, $G$ does not contain a bowtie graph as an induced subgraph. Therefore, by Theorem \ref{th=n}, $\throt(G)=n$.
\end{proof}

This is remarkable because zero forcing has its roots in studying the spectrum of symmetric matrices \cite{AIM}. It is unclear whether the converse to Corollary \ref{spectral_radius} is true. However, Corollary \ref{th_spectral} gives a new line of attack for the converse of Corollary \ref{spectral_radius}. 

\begin{prob}
Let $G$ be a connected graph on $n$ vertices and $m$ edges. Does $\throt(G)=n$ imply that $G$ has maximal spectral radius among connected graphs on $n$ vertices and $m$ edges?
\end{prob}

Another direction for future work is to further study the $\zpf$ throttling number of a graph. Since $\thzpf(G) \leq \thp(G)$ for any graph $G$, a better understanding of $\zpf$ throttling would be useful in obtaining lower bounds for the PSD throttling number of a graph. The \emph{largeur d'arborescence of a graph $G$} (denoted $la(G)$) is defined in \cite{la} as the minimum $k$ such that $G$ is a minor of the Cartesian product of a complete graph on $k$ vertices and a tree. In \cite{Parameters}, it is shown that for any graph $G$, $la(G) = \zpf(G)$. Further research on the types of trees that can show up in the definition of largeur d'arborescence could be useful for studying $\zpf$ propagation and throttling.


Finally, we believe that Lemma \ref{savings} can be generalized for abstract color change rules. This is intuitive since the  relevant pieces in Lemma \ref{savings} can be derived from sets of forces, which exist in the context of an abstract color change rule. Unfortunately, Theorem \ref{finite} does not seem to hold for an abstract color change rule. For example, suppose that the color change rule $R$ is that a blue vertex $u$ can force a white neighbor $w$, if $u$ has at most $4$ neighbors. In this setting, any single vertex is an $R$ forcing set in $K_5$. However, when we consider $K_5$ as an induced subgraph of $K_{10}$, we see that the forcing behavior on the $K_5$ subgraph is influenced by the host graph. Every vertex in the $K_5$ subgraph has $9$ neighbors in the graph as a whole, and may not perform any forces. Furthermore, we cannot choose to color $V(K_{10})\setminus V(K_5)$ blue (as we have done for standard and PSD zero forcing) to recover the forcing behavior of the $K_5$ subgraph. Interestingly, this counterintuitive example invokes a local color change rule. That is, we do not have to look any further than $N[u]$ to determine whether $u$ can perform a force. These considerations motivate us to ask the following questions:

\begin{prob}
What conditions must be imposed on an abstract color change rule $R$ so that 
$\throt_R(G)\geq |V(G)|-k$ and $|V(G)|\geq k$ is a forbidden induced subgraph problem?
What further conditions are necessary to conclude that the set of graphs with $\throt_R(G)\geq |V(G)|-k$ can be characterized by a finite set of forbidden subgraphs?
\end{prob}

\begin{prob}
Does there exists a local abstract color change rule $R$ such that the set of graphs with $\throt_R(G)\geq |V(G)|-k$ and $|V(G)|\geq k$ can be characterized as a forbidden subgraph problem, but no finite set of graphs is a corresponding set of forbidden subgraphs?
\end{prob}

We recognize that these last two questions stray from the linear algebra roots of the zero forcing problem. However, we hope that these questions invite researchers interested in propagation, percolation, or general infection games on graphs to join the conversation. 

\end{section}

\section*{Acknowledgements}
This material is based upon work supported by the National Science Foundation under Grant Number DMS-183991. The authors would also like to thank the referees for their careful reading and helpful comments.

\end{document}